\documentclass[11pt]{article}
\usepackage{latexsym}
\usepackage{amsfonts,amsmath}
\usepackage{epsfig}
\usepackage{latexsym}
\usepackage{lineno}
\usepackage{color}
\usepackage{epsf,amssymb}

\newtheorem{thm}{Theorem}[section]

\newtheorem{cor}[thm]{Corollary}
\newtheorem{lem}[thm]{Lemma}
\newtheorem{prop}[thm]{Proposition}

\newtheorem{claim}{Claim}

\def\vertex(#1){\put(#1){\circle*{2}}}
\def\vertexo(#1){\put(#1){\circle{2}}}
\def\vert(#1){\put(#1){\circle*{1.5}}}
\def\verto(#1){\put(#1){\circle{1.5}}}
\def\lab(#1)#2{\put(#1){\makebox(0,0)[c]{#2}}}
\newenvironment{proof}[1][Proof]{\textbf{#1.} }{\ \rule{0.5em}{0.5em}}

\newcommand{\cart}{\, \Box \,}
\newcommand{\gt}{\gamma_t}

\newcommand{\rk}{\gamma_{{\rm r}k}}
\newcommand{\rd}{\gamma_{{\rm r}2}}

\newcommand{\lex}{\,\circ\,}

\begin{document}

\title{Rainbow domination in the lexicographic product of graphs}

\author{
Tadeja Kraner
\v{S}umenjak \footnote{ FKBV, University of Maribor, Maribor,
Slovenia. The author is also with the Institute  of Mathematics, Physics and
Mechanics, Jadranska 19, 1000 Ljubljana.  email: tadeja.kraner@uni-mb.si}
\and
Douglas F. Rall \footnote{ Herman N. Hipp Professor of Mathematics,
Furman University, Greenville, SC, USA. e-mail:
doug.rall@furman.edu}
\thanks{Supported by a grant from the Simons Foundation
(\# 209654 to Douglas F. Rall). Part of the research done
during a sabbatical visit at the University of Maribor.} \and
Aleksandra Tepeh\footnote{ FEECS, University of Maribor, Maribor,
Slovenia.  The author is also with the Institute  of Mathematics, Physics and
Mechanics, Jadranska 19, 1000 Ljubljana.  email: aleksandra.tepeh@uni-mb.si }  }

\date{}
\maketitle

\begin{abstract}
A \emph{$k$-rainbow dominating function} of a graph $G$ is a map $f$ from
$V(G)$ to the set of all subsets of $\{1,2,\ldots,k\}$ such that
$\{1,\ldots,k\}=\bigcup_{u\in N(v)} f(u)$ whenever $v$ is a vertex
with $f(v)=\emptyset$. The \emph{$k$-rainbow domination number} of $G$ is
the invariant $\rk(G)$, which is the minimum sum (over all the
vertices of $G$) of the cardinalities of the subsets assigned by a
$k$-rainbow dominating function. We focus on the $2$-rainbow
domination number of the lexicographic product of graphs and prove
sharp lower and upper bounds for this number.  In fact, we prove the
exact value of $\rd(G \lex H)$ in terms of domination invariants of
$G$ except for the case when $\rd(H)=3$ and there exists a minimum
$2$-rainbow dominating function of $H$ such that there is a vertex in
$H$ with the label $\{1,2\}$.
\end{abstract}

{\small {\bf Keywords:} domination, total domination, rainbow domination, lexicographic product} \\
\indent {\small {\bf AMS subject classification: 05C69}}

\section{Introduction}

When a graph is used to model  locations or objects  which can
exchange some resource along its edges, the study of ordinary
domination is an optimization problem to determine the minimum
number of locations to store the resource in such a way that each
location either has the resource or is adjacent to one where the
resource resides.  Imagine a computer network in which some of the
computers will be servers and the others clients. There are $k$
distinct resources, and we wish to determine the optimum set of
servers each hosting a non-empty subset of the resources so that any
client (i.e., any computer on the network that is not a server) is
directly connected to a subset of servers that together contain all
$k$ resources. Assuming all resources have the same cost, we seek to
minimize the total number of copies of the $k$ resources. This model
leads naturally to the notion of $k$-rainbow domination.

In general we follow the notation and graph theory terminology
in~\cite{hik-2011}. Specifically, let $G$ be a finite, simple graph
with vertex set $V(G)$ and edge set $E(G)$.
For any vertex $g$ in $G$, the {\em open neighborhood of
$g$}, written $N(g)$, is the set of vertices adjacent to $g$.  The
{\em closed neighborhood} of $g$ is the set $N[g] = N(g) \cup
\{g\}$. If $A\subset V(G)$, then $N(A)$ (respectively, $N[A]$) denotes
the union of open (closed) neighborhoods of all vertices of $A$. (In
the event that the graph $G$ under consideration is not clear we
write $N_G(g)$, and so on.)  Whenever $N[A]=V(G)$ we call $A$ a {\em
dominating set} of $G$.  The {\em domination number} of $G$, denoted
by $\gamma(G)$, is the minimum cardinality of a dominating set of
$G$.  If $G$ has no isolated vertices, then the {\em total
domination number}, $\gamma_t(G)$, is the minimum cardinality of a
{\em total dominating set} in
 $G$ (that is, a subset $S \subseteq V(G)$ such that $N(S)=V(G)$).  It is
clear that $\gamma(G) \le \gamma_t(G) \le 2\gamma(G)$ when
$\gamma_t(G)$ is defined. The {\em maximum degree} of a graph $G$ is
denoted by $\Delta (G)$.

For graphs $G$ and $H$, the {\em Cartesian product} $G \cart H$ is
the graph with vertex set $V(G) \times V(H)$ where two vertices
$(g_1,h_1)$ and $(g_2,h_2)$ are adjacent if and only if either $g_1
= g_2$ and $h_1h_2 \in E(H)$, or $h_1 = h_2$ and $g_1g_2 \in E(G)$.
The {\em lexicographic product} of $G$ and $H$ is the graph $G \circ
H$ with vertex set $V(G) \times V(H)$.  In $G \circ H$ two vertices
$(g_1,h_1)$ and $(g_2,h_2)$ are adjacent if and only if either
$g_1g_2 \in E(G)$, or $g_1 = g_2$ and $h_1h_2 \in E(H)$.   We use
$\pi_G$ to denote the projection map from $G\circ H$ onto $G$
defined by $\pi_G(g,h)=g$.  The projection map $\pi_H$ onto $H$ is
defined in an analogous way.

Fix a vertex $g$ of $G$.  The subgraph of $G \circ H$  induced by
$\{(g,h)\,:\,h \in V(H)\}$ is called an {\em $H$-layer} and is
denoted $^g\!H$.  If $h \in V(H)$ is fixed, then $G\,^h$, the subgraph
induced by $\{(g,h)\,:\, g \in V(G)\}$, is a {\em $G$-layer}.  Note
that every $G$-layer of $G \circ H$ is isomorphic to $G$ and every
$H$-layer of $G\circ H$ is isomorphic to $H$.  It is also helpful to
remember that if $g_1$ and $g_2$ are adjacent in $G$, then the
subgraph of $G\circ H$ induced by $^{g_1}\!H \cup\, ^{g_2}\!H$ is
isomorphic to the join of two disjoint copies of $H$.

For a positive integer $k$ we denote the set $\{1,2,\ldots, k\}$ by
$[k]$. The {\em power set} (that is, the set of all subsets) of
$[k]$ is denoted by $2^{[k]}$. Let $G$ be a graph and let $f$ be a
function that assigns to each vertex a subset of integers chosen
from the set $[k]$; that is, $f \colon V(G) \rightarrow 2^{[k]}$.
The {\em weight}, $\|f\|$, of  $f$ is defined as $\|f\| = \sum_{v\in
V(G)} |f(v)|$. The function $f$ is called a {\em $k$-rainbow
dominating function} of $G$ if for each vertex $v\in V(G)$ such that
$f(v)=\emptyset$ it is the case that
\[
\bigcup_{u\in N(v)} f(u) = \{1,\ldots,k\}\,.
\]

Given a graph $G$, the minimum
weight of a $k$-rainbow dominating function is called the {\em
$k$-rainbow domination number of $G$}, which we denote by $\rk(G)$.

The notion of $k$-rainbow domination in a graph $G$ is equivalent to
domination of the Cartesian product $G \cart K_k$.  There is a
natural bijection between the set of $k$-rainbow dominating
functions of $G$ and the dominating sets of $G \cart K_k$.  Indeed,
if the vertex set of $K_k$ is $[k]$ and $f$ is a $k$-rainbow
dominating function of $G$, then the set
\[
D_f = \bigcup_{v \in V(G)} \Bigl( \bigcup_{i\in f(v)} {\{(v,i)\}}
\Bigr),
\]
 is a dominating set of $G \cart K_k$.  By reversing this one easily sees
 how to complete the one-to-one correspondence.
This proves the following result from~\cite{bhr-2008} where the
concept of rainbow domination was introduced.

\begin{prop}[\rm \cite{bhr-2008}]
\label{rain1} For $k \ge 1$ and for every graph $G$,
$\rk(G)=\gamma(G \cart K_k)$.
\end{prop}

Earlier, Hartnell and Rall had investigated $\gamma(G \cart K_k)$.
See \cite{hr-2004}.  The main focus for them was properties shared
by graphs $G$ for which  $\rd(G)=\gamma(G)$.  In particular, they
proved that for any tree $T$, $\gamma(T) < \rd(T)$. In addition,
they proved a lower bound for $\rk(G)$ that implies $\gamma(G) <
\rk(G)$ for every graph $G$ whenever $k \ge 3$. Expressed in terms
of rainbow domination their result yields the following sharp
bounds.

\begin{thm} [\rm \cite{hr-2004}]
If $G$ is any graph and $k \ge 2$, then
$$\min\{|V(G)|,\gamma(G)+k-2\} \leq \rk(G) \leq k\gamma(G)\,.$$ \label{trivialbds}
\end{thm}

From the  algorithmic point of view, rainbow
domination was first studied in \cite{bhr-2008} where a linear
algorithm for determining a minimum $2$-rainbow dominating set of a
tree was presented. Bre\v{s}ar and Kraner \v{S}umenjak proved that
the $2$-rainbow domination problem is NP-complete even when
restricted to chordal graphs or bipartite graphs \cite{bks-2007}.
Both mentioned results were later generalized for the case of
$k$-rainbow domination problem by Chang, Wu and Zhu \cite{CWZ-2010}.

\section{Upper bounds in general case}

Let $f: V(G) \to 2^{[k]}$ be a $k$-rainbow dominating function of
$G$. For each $A \in 2^{[k]}$ we define $V_A$ by $V_A=\{x \in V(G)\,:\,
f(x)=A\}$ (we will write also $V_{A}^{f}$ to avoid confusion when
more functions are involved). This allows us to speak of the natural
partition induced by $f$ on $V(G)$ instead of working with the
function $f$ itself. For small values of $k$ we may, for example,
abbreviate $V_{\{1,2,3\}}$ by $V_{123}$. Thus, for $k=3$, the
partition of $V(G)$ would be
$(V_{\emptyset},V_1,V_2,V_3,V_{12},V_{13},V_{23},V_{123})$, and for
example, for convenience we say that a vertex in $V_{13}$ is labeled
$\{1,3\}$.  If this partition arises from a $3$-rainbow dominating
function $f$ of minimum weight, then
\[\gamma_{{\rm r}3}(G)=\|f\|=|V_1|+|V_2|+|V_3|+2(|V_{12}|+|V_{13}|+|V_{23}|)+3|V_{123}|\,.\]

\begin{prop} \label{prop:totalupperbound}
For any graph $H$, any graph $G$ without isolated vertices and every
positive integer $k$,
\[ \rk(G\lex H) \le k \gt(G)\,.\]
\end{prop}
\begin{proof}
Fix a vertex $h$ in $H$ and a minimum total dominating set $D$ of $G$.
Define $f : V(G\lex H) \to 2^{[k]}$
by $f((g,x))=[k]$ if $g\in D$ and $x=h$.  Otherwise,
$f((g,x))=\emptyset$. Clearly, $f$ is a $k$-rainbow dominating
function of $G\lex H$ and $\|f\|=k \gt(G)$.  Therefore, $\rk(G\lex
H)\le k\gt(G)$.
\end{proof}

The upper bound from Proposition~\ref{prop:totalupperbound} can be
improved if $H$ has domination number 1.  If $D$ is a minimum
dominating set of $G$ and $h$ is a vertex that dominates all of $H$,
then the partition $V_{[k]}=\{(u,h)\,:\, u \in D\}$ and
$V_{\emptyset}=V(G\lex H)\setminus V_{[k]}$ verifies this improved bound.
\begin{prop}\label{prop:gammaH1}
If  $G$ is any graph without isolated vertices and $H$ is a graph such
that $\gamma(H)=1$, then $\rk(G\lex H)\le k\gamma(G)$.
\end{prop}

In \cite{kpt-2011} the concept of dominating couples was introduced
that enabled the authors to establish the Roman domination number of
the lexicographic product of graphs. We can use that concept to
improve  the upper bound from Proposition \ref{prop:totalupperbound}
in the case $|V(H)|\geq k$.

We say that an ordered couple $(A,B)$ of disjoint sets $A,B
\subseteq V(G)$ is a \textit{dominating couple} of $G$ if for every
vertex $x\in V(G)\setminus B$ there exists a vertex $w \in A\cup B$,
such that $x\in N_{G}(w)$.

\begin{prop}\label{prop:upper}
If $H$ is a graph such that  $|V(H)|\geq k$  and $G$ is a non-trivial
graph, then
$$\rk(G \circ H)\leq \min \{k|A|+\rk(H)|B|: (A,B) \textit{ is a
dominating couple of } G\}.$$
\end{prop}
\begin{proof} Let $(A,B)$ be a dominating couple of $G$.
Let $\widehat{f}$ be a $k$-rainbow dominating function of $H$ with
$\|\widehat{f}\|=\rk(H)$ such that $\bigcup_{v \in V(H)}
\widehat{f}(v)=[k]$ ($\widehat{f}$ exists since $|V(H)|\geq k$). Fix a vertex
$h$ in $H$ and define $f : V(G\lex H) \to 2^{[k]}$ as follows: $f((g,h))=[k]$
if $g\in A$;  $f((g,x))=\emptyset$ if $g \in A$ and $x\neq h$;
$f((g,x))=\widehat{f}(x)$ if $g\in B$ and $x$ is any vertex of $H$;
$f((g,x))=\emptyset$ otherwise. Clearly, $f$ is a $k$-rainbow
dominating function of $G\lex H$. \end{proof}

One can observe that  $(A,\emptyset)$ is a dominating couple if and
only if $A$ is a total dominating set. Thus, if $|V(H)| \geq k$,
Proposition \ref{prop:totalupperbound} is a corollary of Proposition
\ref{prop:upper}.

\medskip

Consider the lexicographic product $P_7 \lex H$, where $P_7$ is a
path of order $7$ and $H$ is a graph consisting of two $4$-cycles
that have one vertex in common. In Figure \ref{potka} this product
is presented in such a way that one can comprehend which $H$-layer
corresponds to which vertex of $P_7$, but we omit edges between
$H$-layers for the reason of clarity. Proposition
\ref{prop:totalupperbound} gives the upper bound $\rd(P_7 \lex
H)\leq 8$ while using the dominating couple $(A,B)=(\{a,b\},\{c\})$
of $P_7$ and $2$-rainbow dominating function of $P_7 \lex H$ depicted
in Figure \ref{potka} we obtain $\rd(P_7 \lex H)\leq 7$ (one can
check that in fact $\rd(P_7 \lex H)=7$).

\begin{figure}[h]
\begin{center}
\includegraphics[height=7cm]{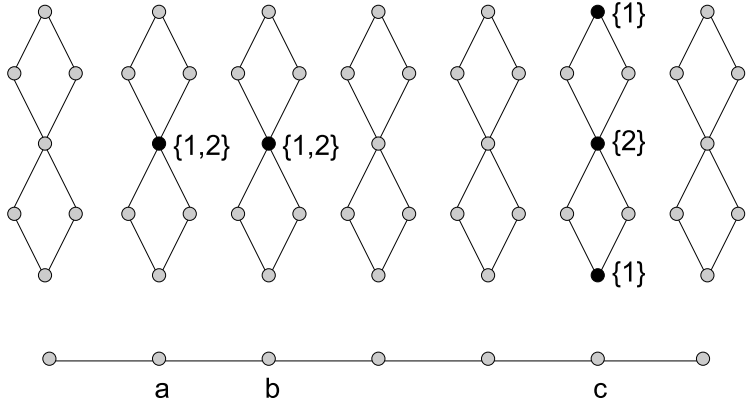}
\caption{Dominating couple of $P_7$ and $2$-rainbow dominating
function of $P_7 \lex H$.\label{potka}}
\end{center}
\end{figure}

\section{$2$-rainbow dominating number}

We now focus on the $k=2$ case and prove that $2\gamma(G)$ is a
lower bound for $\rd(G\circ H)$, unless $H$ has order 1.  If $H$ has
order at least 2 and $G=K_1$, then $\rd(G\circ H)=\rd(H)\ge 2=2\gamma(G)$.
To prove that $2\gamma(G)$ is a lower
bound also when the order of $G$ is greater than 1 we need the
following observations.

\begin{lem}\label{lem:projection1}
Let $G$ and $H$ be non-trivial, connected graphs such that $|V(H)|
\ge 3$ and suppose that $(V_{\emptyset},V_1,V_2,V_{12})$ is the
partition of $V(G\circ H)$ that arises from a $2$-rainbow dominating
function of minimum weight.  It follows that $\pi_G(V_1 \cup
V_{12})$ and $\pi_G(V_2 \cup V_{12})$ each dominate $G$.
\end{lem}
\begin{proof}
Suppose that $A=\pi_G(V_1 \cup V_{12})$ does not dominate $G$. Fix
an arbitrary vertex $x \not\in N_G[A]$. No vertex
of the $H$-layer $^x\!H$ belongs to $V_1\cup
V_{12}$ or is adjacent to a vertex in $V_1\cup
V_{12}$.  This implies that $^x\!H \subseteq V_2$.

Since $H$ is connected and has order at least 3, it follows that $H$
has a vertex of degree 2 or more.  Let $a$ be such a vertex of $H$.
Let $W_{\emptyset}=(V_{\emptyset}\cup(\{x\} \times N_H(a))$, let $W_1=V_1$,
let $W_2=V_2\setminus (\{x\} \times N_H[a])$ and let $W_{12}=V_{12}\cup \{(x,a)\}$.

It is easy to check that $(W_{\emptyset},W_1,W_2,W_{12})$ is a
partition of $V(G\circ H)$ induced by a $2$-rainbow dominating
function and yet $|W_1|+|W_2|+ 2|W_{12}| <  |V_1|+|V_2|+ 2|V_{12}|$.
This contradiction shows that $\pi_G(V_1 \cup V_{12})$ dominates
$G$. Interchanging the roles of 1 and 2 proves the lemma.
\end{proof}

\begin{lem}\label{lem:projection2}
Let $G$ be a non-trivial, connected graph.  There exists a partition
$(V_{\emptyset},V_1,V_2,V_{12})$ of $V(G\circ K_2)$ induced by a
$2$-rainbow dominating function of $G\circ K_2$ of minimum weight such
that  $\pi_G(V_1 \cup V_{12})$ and $\pi_G(V_2 \cup V_{12})$ each
dominate $G$.
\end{lem}
\begin{proof}
As in the proof of Lemma~\ref{lem:projection1} we take a $2$-rainbow
dominating partition $(V_{\emptyset},V_1,V_2,V_{12})$ of $G\circ K_2$
of minimum weight.  Let $A=\pi_G(V_1 \cup V_{12})$ and let $B=\pi_G(V_2 \cup V_{12})$.
Suppose that $A$ does not
dominate $G$.  Let $x$ be a vertex of $G$ not dominated by $A$.  This implies
that $^x\!K_2 \subseteq V_2$.  Let $V(K_2)=\{h_1,h_2\}$.

We define the partition $(W_{\emptyset},W_1,W_2,W_{12})$ as follows.
Let $W_{\emptyset}=V_{\emptyset}\cup \{(x,h_2)\}$, let $W_1=V_1$, let
$W_2=V_2\setminus \{(x,h_1),(x,h_2)\}$ and let $W_{12}=V_{12}\cup\{(x,h_1)\}$.
The partition $(W_{\emptyset},W_1,W_2,W_{12})$ is a $2$-rainbow
dominating partition of $G\circ K_2$.  In addition, $\pi_G(W_1 \cup W_{12})$
dominates more vertices in $G$ than $A$ does  while $\pi_G(W_2 \cup W_{12})$ dominates the
same subset of $G$ that $B$ dominates since $B=\pi_G(W_2 \cup W_{12})$.

If $\pi_G(W_1 \cup W_{12})$ does not dominate $G$, then we can repeat this
process until we arrive at a $2$-rainbow dominating partition such
that the projection onto $G$ of those vertices labeled $\{1\}$ or
$\{1,2\}$ dominates $G$ while simultaneously $B$ is the projection of those
vertices labeled $\{2\}$ or $\{1,2\}$. If $B$ does not dominate $G$, then we
continue on with this new partition and reverse the roles of 1 and 2.  This
will lead to a $2$-rainbow dominating function that has the required property.
\end{proof}

\begin{thm} \label{thm:lower2rainbow}
For every connected graph $G$ and every non-trivial, connected graph
$H$,\\  $\rd(G \lex H) \ge 2\gamma(G)$.
\end{thm}
\begin{proof}
The inequality was proved for $G$ of order 1 at the beginning of
this section. Now, let $|V(G)|\ge 2$ and let
$(V_{\emptyset},V_1,V_2,V_{12})$ be a partition of $V(G \lex H)$
induced by a minimum $2$-rainbow dominating function. Using
Lemmas~\ref{lem:projection1} and \ref{lem:projection2} we derive
\[\rd(G \lex H)=|V_1|+|V_2|+2|V_{12}|\ge |\pi_G(V_1 \cup V_{12})|+
|\pi_G(V_2 \cup V_{12})|\ge 2\gamma(G)\,.\]
\end{proof}

There are large classes of graphs that each have equal domination
number and total domination number.  For example,
see~\cite{DGHM-2006} for a constructive characterization of trees
with this property.  Combining the results in
Proposition~\ref{prop:totalupperbound} and
Theorem~\ref{thm:lower2rainbow} we immediately get the following.

\begin{cor} \label{cor:gammaequalsgammat}
If $G$ and $H$ are non-trivial, connected graphs and
$\gamma(G)=\gamma_t(G)$, then  \\ $\rd(G \lex H)=2\gamma(G)$.
\end{cor}

We now show that with no assumption about the relationship of $\gamma(G)$ and
$\gamma_t(G)$ we  get the same value for the $2$-rainbow domination number
of $G \lex H$ as in Corollary~\ref{cor:gammaequalsgammat}
by instead assuming that $\rd(H)=2$.

\begin{thm}
For every non-trivial, connected graph $G$ and every graph $H$ such that
$\rd(H)=2$,
$$\rd(G \lex H) = 2\gamma(G).$$
\end{thm}
\begin{proof} It is easy to see that if $B \subseteq V(G)$, then $(\emptyset,B)$ is a
dominating couple of $G$ if and only if $B$ is a dominating set of $G$.
Appealing to Proposition~\ref{prop:upper} with $k=2$ shows that
$\rd(G \lex H) \le \rd(H)\gamma(G)= 2\gamma(G)$. The desired equality follows by
Theorem \ref{thm:lower2rainbow}.
\end{proof}

By Proposition \ref{prop:totalupperbound}, an upper bound for $\rd
(G\circ H)$ is $2 \gt(G)$. We will prove that in the case when $\rd
(H)\geq 4$, this bound is actually the exact value for $\rd (G\circ
H)$.  In what follows we say that a layer $^g\!H$ {\it contributes
$k$ (respectively, at least $k$)}  to the weight of a $2$-rainbow
dominating function $f$ of $(G\circ H)$ if $k=\sum_{h \in V(H)}|f(g,h)|$,
(respectively, $k \le \sum_{h \in V(H)}|f(g,h)|$).

\begin{thm}
If $G$ and $H$ are non-trivial, connected graphs and $\rd(H)\geq 4$,
then \\ $\rd (G \circ H)=2\gt (G)$.
\end{thm}
\begin{proof} By the above observation it suffices to prove that $\rd (G\circ H) \geq 2
\gt(G)$. Let $(V_{\emptyset},V_1,V_2,V_{12})$ be the partition of
$V(G\circ H)$ that arises from a $2$-rainbow dominating function $f$
of minimum weight with the property that the cardinality of $\pi_G(
V_{12})$ is maximum.

We claim that $\pi_G(V_1 \cup V_{12})$ and $\pi_G(V_2 \cup V_{12})$
are total dominating sets of $G$ (we already know that they are
dominating sets by Lemma \ref{lem:projection1}). Suppose to the
contrary that one of the sets, say $\pi_G(V_1 \cup V_{12})$, is not
a total dominating set of $G$. It follows that there exists $g\in
\pi_G(V_1 \cup V_{12})$ such that $^{g'}\!H\subseteq
V_{\emptyset}\cup V_2$ for every $g'\in N_G(g)$.
Fix such a neighbor $g'$ of $g$.

Suppose that $f((g,x))\neq \emptyset$ for every vertex $x$ in $H$.
Since $\rd(H)\geq 4$ we see that $^g\!H$ contributes at least 4 to the
weight of $f$.  Let $h$ be any vertex of $H$, and  let
$\widehat{f}$ be the function on $V(G\lex H)$, induced by the
partition $(W_{\emptyset},W_1,W_2,W_{12})$ of $V(G\circ H)$, where
$W_{\emptyset}=V_{\emptyset}\cup ((^g\!H \cup\, ^{g'}\!H) \setminus \{(g,h), (g',h)\})$,
$W_1=V_1\setminus (^g\!H \cap V_1)$, $W_2=V_2\setminus ((^g\!H \cup\, ^{g'}\!H)\cap V_2)$
and $W_{12}=(V_{12}\setminus ^g\!H) \cup \{(g,h),(g',h)\}$. One can check that
$\widehat{f}$ is a $2$-rainbow dominating function of $G\circ H$, and
by its definition $\|\widehat{f}\| \le \|f\|$.  This implies that
$\widehat{f}$ is a $2$-rainbow dominating function of $G\circ H$ of minimum
weight.  This is a contradiction since $|\pi_G(W_{12})| > |\pi_G(V_{12})|$.
It follows that $V_{\emptyset}\cap\, ^g\!H\neq \emptyset$.

Now we distinguish the following two possibilities.

Case 1. If every vertex from $V_{\emptyset}\cap\, ^g\!H$ is adjacent
to a vertex in $(V_2\cup V_{12})\cap\, ^g\!H$ (i.e. the $2$-rainbow
domination of the $^g\!H$-layer is assured within the layer), then
the $^g\!H$-layer contributes at least 4 to the weight of $f$, since
$\rd(H)\geq 4$. By defining $\widehat{f}$ as above we obtain a
$2$-rainbow dominating function on $V(G\circ H)$, the weight of
which is less than or equal to the weight of $f$, a contradiction in
either case (in the second case with $f$ being a $2$-rainbow
dominating function with the maximum cardinality of $\pi_G(
V_{12})$).

Case 2. Suppose that there exists $(g,h)\in V_{\emptyset}\cap\,
^g\!H$ that is not adjacent to any vertex in $(V_2\cup V_{12})\cap\,
^g\!H$ (and is adjacent to $(g,h'')$ with $f((g,h''))=\{1\}$). This
implies that there exist $g'\in N_G(g)$ and $h'\in V(H)$ such that
$f((g',h'))=\{2\}$.

First we show that there are at least two vertices in $(V_1\cup
V_{12})\cap\, ^g\!H$. If we assume to the contrary  that $|(V_1\cup
V_{12})\cap\, ^g\!H|=1$, then $V_{12}\cap\, ^g\!H=\emptyset$ and there
exists $(g,h''')$  with $f((g,h'''))=\{2\}$ (otherwise $H$ would
have a universal vertex, but this is in contradiction with
$\rd(H)\geq 4$). Moreover, there are at least two such vertices,
since equalities $|V_{1}\cap\, ^g\!H|=|V_{2}\cap\, ^g\!H|=1$ imply
$\rd(H)\leq 3$, a contradiction. Let
$(W_{\emptyset},W_1,W_2,W_{12})$ be the following partition of
$V(G\circ H)$: $W_{\emptyset}=V_{\emptyset} \cup \{(g,h''')\}$,
$W_1=V_1$, $W_2=V_2\setminus \{(g,h'''),(g',h')\}$ and
$W_{12}=V_{12}\cup \{(g',h')\}$. One can observe that this partition
induces a $2$-rainbow dominating function $\widehat{f}$ on $V(G\circ
H)$ with the same weight as $f$, and $|\pi_G(W_{12})|>|\pi_G(V_{12})|$, a contradiction. Hence
there are at least two vertices in $(V_1\cup V_{12})\cap\, ^g\!H$.

Now, the function $\widehat{f}$ induced by the partition
$(W_{\emptyset},W_1,W_2,W_{12})$ where $W_{\emptyset}=V_{\emptyset}
\cup \{(g,h'')\}$, $W_1=V_1\setminus \{(g,h'')\}$, $W_2=V_2\setminus
\{(g',h')\}$ and $W_{12}=V_{12}\cup \{(g',h')\}$ is a $2$-rainbow
dominating function on $V(G\circ H)$ with the same weight as $f$,
and such that $|\pi_G(W_{12})|>|\pi_G(V_{12})|$,
which is a final contradiction. The claim that $\pi_G(V_1 \cup
V_{12})$ and $\pi_G(V_2 \cup V_{12})$ are both total dominating sets
of $G$ is proved.

From this we easily derive the desired result. Namely, since both
of $\pi_G(V_1 \cup V_{12})$ and $\pi_G(V_2 \cup V_{12})$ are total
dominating sets of $G$ we get
\[\rd (G\circ H)=|V_1|+|V_2|+2|V_{12}|\geq |\pi_G(V_1 \cup V_{12})|+
|\pi_G(V_2 \cup V_{12})|\geq 2\gamma_t(G)\,.\] \end{proof}

\bigskip

In the cases $\rd(H)=2$ and $\rd(H)\geq 4$ we obtained the exact
values for $\rd (G\lex H)$. However, the case $\rd(H)=3$ is the most
challenging. Combining Proposition~\ref{prop:upper} and
Theorem~\ref{thm:lower2rainbow} we obtain the following sharp
bounds.
\begin{cor} \label{cor:gen2bounds}
If $G$ and $H$ are non-trivial, connected graphs such that
$\rd(H)=3$,  then
\[2\gamma(G) \leq \rd(G\lex H) \leq \min \{2|A|+3|B|: (A,B) \textit{ is a
dominating couple of } G\}\,.\]
\end{cor}

As we will see in the theorem that follows, the upper bound in the
above corollary is actually the exact value provided that every
minimum $2$-rainbow dominating function of $H$ enjoys a certain
property.

\begin{thm} \label{challenge}
Let $H$ be a connected graph with $\rd(H)=3$ and assume that for
every minimum $2$-rainbow dominating function $\varphi$ of $H$,
$\varphi(h)  \neq \{1,2\}$ for every vertex $h$ of $H$.  If $G$ is
any graph, then
\[\rd (G\circ H)=\min \{2|A|+3|B|: (A,B) \textit{ is a
dominating couple of }  G \}\,.\]
\end{thm}
\begin{proof}
If $G$ is isomorphic to $K_1$, the claim from the theorem obviously
holds. Hence we assume that $G$ is a non-trivial graph. The graph $H$
contains at least four vertices since no connected graph of order less
than 4 has 2-rainbow domination number 3. Since no minimum weight
2-rainbow dominating function of $H$ uses the label $\{1,2\}$, it
follows that every minimum weight 2-rainbow dominating function of $H$
uses both of $\{1\}$ and $\{2\}$.

From among all minimum $2$-rainbow dominating functions of the graph
$G\circ H$ assume that $f$ is chosen with the property that for
every minimum $2$-rainbow dominating function $f_1$ of $G\circ H$,
\[\left |\,\{x \in V(G)\,:\, \{1,2\}=\bigcup_{h\in V(H)}f(x,h)\,\}\,\right | \ge
\left|\,\{x \in V(G)\,:\, \{1,2\}=\bigcup_{h\in
V(H)}f_1(x,h)\,\}\,\right|\,.\]

Let $(V_{\emptyset},V_1,V_2,V_{12})$ be the partition of $V(G \circ H)$ induced by $f$.

We now define a partition $(W_{\emptyset},W_1,W_2,W_{12})$ of
$V(G)$.
\begin{itemize}
\item $W_{\emptyset}=\{\, w \in V(G) \,:\, ^w\!H\subseteq V_{\emptyset}\,\}$;
\item $W_1=\{\, w\in V(G)\,:\, ^w\!H\subseteq V_1 \cup V_{\emptyset} \textup{ and } ^w\!H\cap V_1 \neq \emptyset\,\}$;
\item $W_2=\{\, w\in V(G)\,:\, ^w\!H\subseteq V_2 \cup V_{\emptyset} \textup{ and } ^w\!H\cap V_2 \neq \emptyset\,\}$; and
\item $W_{12}=\{\, w\in V(G)\,:\, \{1,2\}=\bigcup_{h\in V(H)}f(w,h)\,\}$.
\end{itemize}

First we prove that if $w \in W_1\cup W_2$, then by the choice of
the $2$-rainbow dominating function $f$ the layer $^w\!H$ contributes
exactly 1 to the weight of $f$.

\begin{claim} \label{claimW1}
If $x \in W_1$, then there is exactly one vertex in $^x\!H$ labeled
$\{1\}$. If $x \in W_2$, then there is exactly one vertex in $^x\!H$
labeled $\{2\}$.
\end{claim}

Fix $x \in W_1$, and suppose there are distinct vertices $h_1$ and
$h_2$ in $H$ such that $f(x,h_1)=\{1\}=f(x,h_2)$.

If $x$ is isolated in $G$, then $f$ restricted to $^x\!H$ is a minimum weight
2-rainbow dominating function of $^x\!H$.  However, since $x \in W_1$,
it follows that every vertex in $^x\!H$ is labeled $\{1\}$, and so
$^x\!H$ contributes at least 4 to the weight of $f$.  This contradiction
shows that $x$ is not isolated in $G$.  Let $v\in N_G(x)$.

We claim that there is a vertex with label $\emptyset$ in $^{x}\!H$.
Suppose to the contrary that $^{x}\!H \subseteq V_1$. We infer that such an $f$ cannot be a
minimum weight $2$-rainbow dominating function, since we can obtain a
$2$-rainbow dominating function of weight less than $\|f\|$ by
replacing the label $\{1\}$ with $\emptyset$ on each vertex in $^{x}\!H$ except for one
 and relabeling one vertex in $^{v}\!H$ with $\{1,2\}$.
Because of this contradiction it follows that at least one vertex
in $^{x}\!H$ has label $\emptyset$.

Hence, there is a vertex $y$ adjacent to $x$ in
$G$ such that $^y\!H$ contains a vertex labeled $\{2\}$ or a vertex
labeled $\{1,2\}$. However, by the minimality of the weight of $f$
it follows that $\{1,2\} \neq \bigcup_{h\in V(H)}f(y,h)$ (for otherwise
we could ``relabel'' $(x,h_2)$ with $\emptyset$ and the result would
be a $2$-rainbow dominating function with  smaller weight). Thus,
suppose that $f(y,h_3)=\{2\}$ for some $h_3\in V(H)$. Let $\widehat{f}:
V(G\circ H)\to 2^{[2]}$ be defined by
$\widehat{f}(x,h_2)=\emptyset$, $\widehat{f}(y,h_3)=\{1,2\}$ and
$\widehat{f}(g,h)=f(g,h)$ for every other vertex $(g,h) \not\in
\{(x,h_2),(y,h_3)\}$.  It is easy to see that $\widehat{f}$ is a
$2$-rainbow dominating function of $G\circ H$ having the same weight
as $f$.  However, under this new $2$-rainbow dominating function,
$\widehat{f}$, there are more $H$-layers containing both of the
labels 1 and 2 than there are with $f$.  This contradiction proves the
first statement.  The second statement is proved by interchanging
the roles of 1 and 2. Hence, the claim is verified.

\medskip
Note that $\rd(H)=3$ implies that $\gamma(H)=2$ or $\gamma(H)=3$.
Hence, if $w \in W_1 \cup W_2$, then Claim~\ref{claimW1} implies
that there exist $u,v\in N_G(w)$ and $h,k \in V(H)$ such that $1\in
f(u,h)$ and $2 \in f(v,k)$ (where it is possible that $u=v$ or
$h=k$).  It also follows from this claim that if $w \in V(G)$ and
the layer $^w\!H$ contributes 2 or more to the weight of $f$, then
$w \in W_{12}$. Indeed, $w \in W_{12}$ if and only if $^w\!H$
contributes at least 2 to the weight of $f$. Furthermore, suppose $w
\in W_{12}$ and $^w\!H$ contributes 3 or more to the weight of $f$.
Since $\rd(H)=3$ and $f$ is a minimum $2$-rainbow dominating function
of $G\circ H$, for this $w$ we can assume that the restriction of
$f$ to $^w\!H$ is a $2$-rainbow dominating function of the subgraph of
$G\circ H$ induced by this $H$-layer.  This follows by using the
labels $\emptyset$, $\{1\}$ and $\{2\}$ in a minimum $2$-rainbow
dominating function $\varphi$ of $H$ and setting
$f(w,h)=\varphi(h)$. We conclude that if $w \in W_{12}$, then the
layer $^w\!H$ contributes precisely 2 or 3 to the weight of $f$.

We will need to be able  to distinguish the following types of
$H$-layers:
\begin{itemize}
\item The layer $^x\!H$ is of {\it Type 1} if for some $y \in N_G(x)$, $y\in W_{12}$.
\item The layer $^x\!H$ is of {\it Type 2} if $^x\!H$ is not of Type 1,
and there exist distinct vertices $y,z \in N_G(x)$ such that $y \in
W_1$ and $z\in W_2$.
\item The layer $^x\!H$ is of {\it Type 3} if $^x\!H$
is $2$-rainbow dominated by $f$ restricted to $^x\!H$.
\end{itemize}

First we prove the following claim.

\begin{claim} \label{threetypes}
For any vertex $x$ of $G$, the $H$-layer $^x\!H$ is of exactly one of the
types listed above.
\end{claim}

To prove the claim we use the fact that $(W_{\emptyset},W_1,W_2,W_{12})$
is a partition of $V(G)$.

Consider first a vertex $x$ in $W_{\emptyset}$.  Every vertex in $^x\!H$ must have a
neighbor with a 1 in its label and a neighbor with a 2 in its label.  That is,
either $x$ has a neighbor in $W_{12}$, or $x$ has a neighbor in $W_1$ and
a neighbor in $W_2$.  In other words, $^x\!H$ is of Type 1 or of Type 2.

Suppose $x \in W_1$.  By Claim~\ref{claimW1} there is a unique $h \in V(H)$ such
that $f(x,h)=\{1\}$, and $f(x,h')=\emptyset$ for every $h' \in V(H)\setminus \{h\}$.  Since
$\gamma(H) \ge 2$ there is some vertex, say $(x,k)$, in $^x\!H$ that is not adjacent
to $(x,h)$ and such that  $f(x,k)=\emptyset$.  Now,
$(x,k)$ must have a neighbor with a 1 in its label and a neighbor with a 2 in its label.  That is,
either $x$ has a neighbor in $W_{12}$, or $x$ has a neighbor in $W_1$ and
a neighbor in $W_2$.  In other words, $^x\!H$ is of Type 1 or of Type 2.
The case $x \in W_2$ is handled similar to this with the roles of 1 and 2 interchanged.

Finally, assume that $x \in W_{12}$.  The layer $^x\!H$ contributes either 2 or 3 to
the weight of $f$.  Suppose $^x\!H$ contributes 3.  By an earlier argument we may assume
that $f$ restricted to $^x\!H$ is a 2-rainbow dominating function of $^x\!H$.
By our assumption on $H$ this means that the label $\{1,2\}$ does not appear on any
vertex in $^x\!H$.  Assume that $^x\!H$ is of Type 1 or Type 2.  In this case one of the labels
in $^x\!H$ could be changed from $\{1\}$ to $\emptyset$ or from $\{2\}$ to $\emptyset$
(whichever one of $\{1\}$ or $\{2\}$ that occurs twice in $^x\!H$) and this would yield a 2-rainbow
dominating function of $G \circ H$ having smaller weight than $f$.  This contradiction
implies that an $H$-layer of Type 3 is not also of Type 1 or Type 2.

Assume that $^x\!H$ contributes exactly 2 to the weight of $f$.  If $\{1,2\}$ occurs
as a label on some vertex $(x,h)$ in $^x\!H$, then since $\gamma(H) \ge 2$ there is
some vertex $(x,k)$  in $^x\!H$ that is not adjacent to $(x,h)$.  This vertex $(x,k)$
has a neighbor with a 1 in its label and a neighbor with a 2 in its label. These neighbors
are not in $^x\!H$ and once again as above we conclude that $^x\!H$ is of Type 1 or of Type 2.

Now, suppose that $^x\!H$ contributes exactly 2 to the weight of $f$ and there exist
distinct vertices $h_1$ and $h_2$ in $H$ such that $f(x,h_1)=\{1\}$ and $f(x,h_2)=\{2\}$.
Suppose that $^x\!H$ is not of Type 1 nor of Type 2. Since $\rd(H)=3$ there
is a vertex $(x,h)$ that is not adjacent to both $(x,h_1)$ and
$(x,h_2)$.  If $(x,h)$ is adjacent to neither of them, then it has a neighbor
with a 1 in its label and a neighbor with a 2 in its label and
both of these neighbors lie outside of $^x\!H$.  However, this contradicts our assumption
that $^x\!H$ is not of Type 1 nor of Type 2.  Thus, we may assume without
loss of generality that $(x,h)$ is adjacent
to $(x,h_1)$ but not to $(x,h_2)$. It follows that there exists $x'\in N_G(x)$
such that $x'\in W_2$. Let $f(x',h')=\{2\}$. By Claim~\ref{claimW1}, $f(x',k)=\emptyset$
for every $k\in V(H)\setminus \{h'\}$.  Since $^x\!H$ is not of Type 1 nor of Type 2, no neighbor
of $x$ belongs to $W_1\cup W_{12}$. This means that every vertex in
$^x\!H\setminus \{(x,h_1),(x,h_2)\}$ is adjacent to $(x,h_1)$. Let
$g$ be the function defined on $V(H)$ by $g(h_1)=\{1,2\}$,
$g(h_2)=\{2\}$ and $g(v)=\emptyset$ for every other vertex  $v$ of $H$. This function $g$ is
a $2$-rainbow dominating function of $H$ having weight 3 and also having
a vertex labeled $\{1,2\}$.  This contradiction shows that $^x\!H$ is of
Type 1 or of Type 2 and finishes the proof of Claim~\ref{threetypes}.

\medskip
We may assume without loss of generality that $|W_1| \ge |W_2|$.  We
now modify the function $f$ to produce another minimum $2$-rainbow
dominating function $p$ of $G\circ H$ which has the property that
each $H$-layer that receives a non-empty label contributes either
exactly 2 or exactly 3 to the weight of $p$. The general idea is
that if $w\in W_{\emptyset} \cup W_{12}$, then the labeling under
$p$ for vertices in $^w\!H$ will be the same as it was under $f$.
Thus, all $H$-layers that contribute 2 or 3 to the weight of $f$
will also contribute that amount to the weight of $p$. On the other
hand, some $H$-layers that contribute 1 to the weight of $f$ will
contribute 2 to the weight of $p$ while others will contribute 0 to
the weight of $p$.

We define $p$ by specifying the partition $(U_{\emptyset}, U_1,U_2,
U_{12})$ that $p$ induces on $V(G \circ H)$.  Let
\[U_1=\{\,(w,k)\,:\, w \in W_{12} \textup{ and } f(w,k)=\{1\}\,\}\,,\]
\[U_2=\{\,(w,k)\,:\, w \in W_{12} \textup{ and } f(w,k)=\{2\}\,\}\,,\]
\[U_{\emptyset}= V_{\emptyset} \cup \{\,(w,k)\,:\, w\in W_1 \,\}\,,\]
and
\[U_{12}= V_{12} \cup  \{\,(w,k)\,:\, w \in W_2 \textup{ and }f(w,k)=\{2\}\,\}\,.\]

To prove that $p$ is a $2$-rainbow dominating function of $G\circ H$
let $(g,h)\in U_{\emptyset}$ (in other words, $(g,h)$ is such that
$f(g,h)=\emptyset$, or $g \in W_1$ and $f(g,h)=\{1\}$). All
possibilities are covered in the following cases.

\begin{itemize}
\item Suppose $^g\!H$ is of Type 1.  As noted above, $g$ has a neighbor $g'\in W_{12}$.
The vertex $(g,h)$ is adjacent to every vertex in $^{g'}\!H$.  By
the definitions of $U_1$, $U_2$, and $U_{12}$ it follows that
$(g,h)$ has a neighbor that (under $p$) contains 1 in its label and
a neighbor that (under $p$) contains 2 in its label.

\item Suppose $^g\!H$ is of  Type 2. Now, $g$ has a neighbor $y\in W_1$ and a neighbor $z \in W_2$.
There exists $h'\in V(H)$ such that $f(z,h')=\{2\}$. By the
definition of $p$, $p(z,h')=\{1,2\}$ and $(g,h)$ is adjacent to
$(z,h')$.

\item Suppose that $^g\!H$ is of Type 3.  By the definition of $p$
the labels in $^g\!H$ under $p$ agree with those under $f$, and by
our assumption about $f$, the vertex $(g,h)$ with the property
$f(g,h)=\emptyset$ has a neighbor in $^g\!H \cap U_1$ and a neighbor
in $^g\!H \cap U_2$.

\end{itemize}

Hence we see that in all of the above,
$\{1,2\}=\bigcup\{\,p(g',h')\,:\,(g',h')\in N(g,h)\}$. It follows
that $p$ is a $2$-rainbow dominating function of $G\circ H$, and by
its definition $\|p\| \le \|f\|$.  Therefore, $\|p\|=\rd(G\circ H)$.

Let
\[A=\{\,x\in V(G)\,:\,  ^x\!H \textup{ contributes 2 to the weight of } p\,\}\,,\quad  {\rm and}\]
\[B=\{\,x\in V(G)\,:\,  ^x\!H \textup{ contributes 3 to the weight of } p\,\}\,.\]

The definition of $p$ shows that $\|p\|=2|A|+3|B|$.  It remains to show
that $(A,B)$ is a dominating couple of $G$.  For this purpose let $g
\in V(G)\setminus B$.  If $g$ does not belong to $A$, then $^g\!H \subseteq
U_{\emptyset}$.  Since $p$ is a $2$-rainbow dominating function it
follows that $g$ has a neighbor in $A \cup B$.  Finally, assume that
$g\in A$.  This means that $^g\!H$ contributes 2 to the weight of
$p$. Since $\rd(H)=3$ there exists at least one vertex $(g,h)\in
U_{\emptyset}$ such that $\{1,2\}\neq \bigcup\{ p(g,k)\,:\, k\in
N_H(h)\}$.  (That is, $(g,h)$ is not $2$-rainbow dominated by $p$ from
within $^g\!H$.)  It follows that $(g,h)$ has a neighbor in some
$^{g'}\!H$ such that $^{g'}\!H \subseteq \left( (A\setminus \{g\}) \cup B
\right) \times V(H)$. Hence $g$ and $g'$ are adjacent in $G$, and
$g'\in A \cup B$.  Therefore, $(A,B)$ is a dominating couple of $G$.
\end{proof}

The factors of the lexicographic product represented in Figure
\ref{potka} satisfy the conditions of the above theorem so this
graph attains the upper bound of Corollary~\ref{cor:gen2bounds}.

We were also able to improve the upper bound from this corollary in
the case of the lexicographic product of paths and graphs $H$ that
do not satisfy the condition on $H$ from Theorem \ref{challenge}.
We would like to point out that the
construction used in the proof of the next proposition enabled us
also to find a family of graphs that attains the lower bound
$2\gamma(G)$.

\begin{prop}\label{prop:G-path}
Let $H$ be a connected graph with $\rd(H)=3$ and the property that
there exists a $2$-rainbow dominating function of $H$ of minimum
weight such that there is a vertex in $H$ labeled with $\{1,2\}$. It
follows that
$$\rd (P_n\circ H)\leq \left\{ \begin{array}{ccc}
           6\lfloor \frac{n}{7} \rfloor +k & , & n \equiv k \pmod{7} \textup{ for } k=0,3,4,5,6,\\
            &&\\
            6\lfloor \frac{n}{7} \rfloor +k+1 & , & n \equiv k \pmod{7} \textup{ for }
            k=1,2.
            \end{array} \right.$$
\end{prop}
\begin{proof}
Let $H$ be a connected graph with $\rd(H)=3$ and suppose there
exists a $2$-rainbow dominating function $f$ of $H$ of minimum weight
such that $V_{12} \neq \emptyset$. Let $u,v\in V(H)$ be the
vertices with $f(u)=\{1,2\}$ and (without loss of generality)
$f(v)=\{1\}$.

To end the proof it suffices to construct a $2$-rainbow dominating
function $p$ on $P_n \circ H$ with desired weight for each case. We
will represent $p$ with a table of integers $0,1,2,3$ where these
numbers denote subsets $\emptyset, \{1\}, \{2\}$ and $\{1,2\}$,
respectively. Numbers in the first and second row correspond to the
values of $p$ in the $^u\!P_n$-layer and $^v\!P_n$-layer,
respectively (we omit other $P_n$-layers, since only zeros appear in
them).

One can check that for each $i=2,3,4,5,6,7,8$, $R_i$ depicted below
represents a $2$-rainbow dominating function on $P_i \circ H$.

\begin{tabular}{ccccccc}

\begin{tabular}[t]{|c|}
\multicolumn{1}{c}{$R_2$}\\\hline
 30\\\hline
 10
\\\hline
\end{tabular} &

\begin{tabular}[t]{|c|}
\multicolumn{1}{c}{$R_3$}\\\hline
 030\\\hline
 010
\\\hline
\end{tabular} &

\begin{tabular}[t]{|c|}
\multicolumn{1}{c}{$R_4$}\\\hline
 0330\\\hline
 0000
\\\hline
\end{tabular} &

\begin{tabular}[t]{|c|}
\multicolumn{1}{c}{$R_5$}\\\hline
 02120\\\hline
 01010
\\\hline
\end{tabular} &

\begin{tabular}[t]{|c|}
\multicolumn{1}{c}{$R_6$}\\\hline
 030030\\\hline
 010010
\\\hline
\end{tabular} &

\begin{tabular}[t]{|c|}
\multicolumn{1}{c}{$R_7$}\\\hline
 0210210\\\hline
 0100020
\\\hline
\end{tabular} &

\begin{tabular}[t]{|c|}
\multicolumn{1}{c}{$R_8$}\\\hline
 02102130\\\hline
 01000200
\\\hline
\end{tabular}

\end{tabular}

\medskip

To construct a $2$-rainbow dominating function of $P_n \circ H$ for
$n\geq 9$ we distinguish three cases.

If $n\equiv 0 \pmod{7}$, i.e $n=7t$ for some integer $t\geq 1$, then
we obtain the table that corresponds to a desired function by taking
$t$ copies of $R_7$.

\begin{tabular}[t]{|c|}
\multicolumn{1}{c}{}\\\hline  $\underbrace{R_7R_7 \ldots R_7}_{t}$
\\\hline
\end{tabular}

If $n=7t+1$ for some $t\geq 1$, then we take $t-1$ copies of $R_7$
and one copy of $R_8$.

\begin{tabular}[t]{|c|c|}
\multicolumn{2}{c}{}\\\hline  $\underbrace{R_7R_7 \ldots R_7}_{t-1}$
&$R_8$
\\\hline
\end{tabular}

For all other cases (when $n=7t+i$, for $t\geq 1$ and $2\leq i \leq
6$), we take $t$ copies of $R_7$ and one copy of $R_i$.

\begin{tabular}[t]{|c|c|}
\multicolumn{2}{c}{}\\\hline  $\underbrace{R_7R_7 \ldots R_7}_{t}$ &
$R_i$
\\\hline
\end{tabular}

Verification that in each case we obtain a $2$-rainbow dominating
function of desired weight is left to the reader.
\end{proof}

As we have seen in the proof, $\rd (P_7 \lex H)=6=2\gamma (P_7)$, so
the lower bound in Corollary~\ref{cor:gen2bounds} is attained. Using
similar ideas we can also construct an infinite family of graphs
that attain this lower bound.

Let $H$ be a connected graph with $\rd(H)=3$ and the property that
there exists a $2$-rainbow dominating function of $H$ of minimum
weight such that $V_{12}\neq \emptyset$. As above, let $u,v\in V(H)$
be the vertices with $f(u)=\{1,2\}$ and, say $f(v)=\{1\}$. Let $G$
be a graph obtained from  $m$ paths isomorphic to $P_6$ and $n$
paths isomorphic to $P_2$ in such way that we glue them together
along a pendant vertex in each path, see Figure~\ref{tree}. In this
figure the tables as above represent the values of a $2$-rainbow
dominating function on $G \lex H$ only in the $G\,^u$-layer (above)
and $G\,^v$-layer (below), since only zeros appear elsewhere. This
construction gives us $\rd (G \lex H)\leq 4m+2$. On the other hand,
one can verify that $\gamma(G)= 2m+1$. Thus, by
Corollary~\ref{cor:gen2bounds}, $\rd (G \lex H)=2 \gamma(G)$.

\begin{figure}[h]
\begin{center}
\includegraphics[height=7cm]{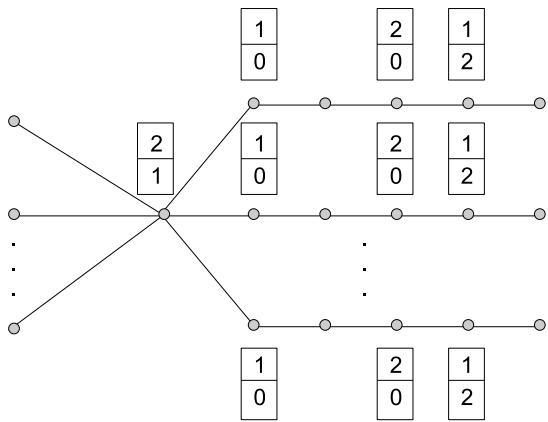}
\caption{$2$-rainbow dominating function of $G \lex H$.\label{tree}}
\end{center}
\end{figure}

\section{Concluding remarks}

By Proposition \ref{prop:G-path}, $\rd(P_5 \lex P_4)\leq 5$, while
the lower and the upper bounds from Corollary \ref{cor:gen2bounds}
are 4 and 6, respectively. In fact, it is a matter of case analysis
to show that $\rd(P_5 \lex P_4)=5$, but it is our conjecture that
the bound from Proposition \ref{prop:G-path} is actually the exact
value.

More generally, it remains an open problem to find the formula for
$\rd(G \lex H)$ in the case when $\rd(H)=3$ and there exists a
minimum $2$-rainbow dominating function of $H$ such that there is a
vertex in $H$ with the label $\{1,2\}$.

\section{Acknowledgements}

We thank the anonymous referees for a very careful reading of our 
manuscript and for a number of suggestions that helped to clarify the presentation.

\end{document}